\documentclass[pdflatex,sn-mathphys-num]{sn-jnl}% Math and Physical Sciences Numbered Reference Style
\usepackage{graphicx}%
\usepackage{multirow}%
\usepackage{amsmath,amssymb,amsfonts}%
\usepackage{amsthm}%
\usepackage{mathrsfs}%
\usepackage[title]{appendix}%
\usepackage{xcolor}%
\usepackage{textcomp}%
\usepackage{manyfoot}%
\usepackage{booktabs}%
\usepackage{algorithm}%
\usepackage{algorithmicx}%
\usepackage{algpseudocode}%
\usepackage{listings}%
\usepackage{tikz}
\usepackage{epsfig}
\usepackage{latexsym}
\usepackage{color}
\usepackage{caption}
\usepackage{yfonts}
\usepackage{hyperref}
\usepackage{ifthen}
\usepackage{float} %1 te tri vrstice potrebne, da je caption nad tabelo
\floatstyle{plaintop}%2
\restylefloat{table} %3

\usepackage{enumitem}

\usepackage[table]{xcolor}
\definecolor{FormOneLimited}{RGB}{0,128,0}     % zelena
\definecolor{FormTotal}{RGB}{0,90,180}         % modra
\definecolor{FormCode}{RGB}{210,120,20}        % oranžna
\definecolor{GoalGray}{RGB}{230,230,230}

\theoremstyle{thmstyleone}%
\newtheorem{theorem}{Theorem}%  meant for continuous numbers
\newtheorem{proposition}[theorem]{Proposition}% 
\newtheorem{lemma}[theorem]{Lemma}

\newtheorem{corollary}[theorem]{Corollary}

\newtheorem{problem}[theorem]{Problem}

\theoremstyle{thmstyletwo}%

\theoremstyle{thmstylethree}%

\newcommand{\gammakL}{\gamma_k^{\mathrm L}}
\newcommand{\gammaenaL}{\gamma_1^{\mathrm L}}
\tikzset{My Style/.style={draw, circle, fill=black,scale=0.5}} %za risanje grafov-stil vozlišè
\tikzset{My Style2/.style={draw, circle, fill=white,scale=0.5}} %za risanje grafov-stil vozlišè

\begin{document}

\title[Article Title]{On $k$-limited domination: complexity and Cartesian products}


\author*[1,2,3]{\fnm{Aleksandra} \sur{Tepeh}}\email{aleksandra.tepeh@um.si}

\affil*[1]{\orgname{University of Maribor, FEECS}, \orgaddress{\street{Koro\v{s}ka cesta 46}, \city{Maribor}, \postcode{2000}, \country{Slovenia}}}

\affil[2]{\orgname{Institute of Mathematics, Physics and Mechanics}, \orgaddress{\street{Jadranska 19}, \city{Ljubljana}, \postcode{1000}, \country{Slovenia}}}

\affil[3]{\orgname{University of Maribor, FNM}, \orgaddress{\street{Koro\v{s}ka cesta 160}, \city{Maribor}, \postcode{2000}, \country{Slovenia}}}


\abstract{A dominating set is called $k$-limited if every vertex in the set has at most $k$ neighbors outside it. The minimum cardinality among all $k$-limited dominating sets of a graph $G$ is the $k$-limited domination number, denoted by $\gammakL(G)$.
We prove that, for every fixed integer $k\ge 2$, deciding whether a graph admits a $k$-limited dominating set of size at most $\ell$ is $\mathsf{NP}$-complete.
In addition, a systematic study of $k$-limited domination in Cartesian products is initiated. In particular, we establish general lower and upper bounds for $\gammakL(G\square H)$, show that both are sharp, and derive exact values for several natural families of graph products. 
}




\maketitle

\section{Introduction}

A dominating set of a graph $G$ is a set of vertices such that every vertex outside the set has a neighbor in it, and the minimum size of such a set is the domination number $\gamma(G)$. Over the years, domination has been studied from many different perspectives and on many graph classes; we refer to the monographs \cite{HHH1,HHH2,HHH3} for background and further references. At the same time, several central open problems in domination theory remain unresolved, among them Vizing's conjecture on domination in Cartesian products \cite{BB}, and the problem of determining the domination number of hypercubes, which is known exactly only in a limited number of cases \cite{HarLiv,Wee,Wu}.

Recently, in \cite{podgorica1}, the notion of $k$-limited domination was introduced as a natural variant of classical domination. Many previously studied domination variants impose additional requirements on vertices outside the dominating set; for instance, in \emph{$k$-domination} every vertex outside the set must have at least $k$ neighbors in it \cite{FinkJacobson}, while in \emph{$[1,k]$-domination} every vertex outside the set must have at least one and at most $k$ neighbors in it \cite{Chellali,Peterin}. By contrast, $k$-limited domination places the restriction on the vertices inside the dominating set. In the usual domination model, one seeks a dominating set of minimum cardinality, but no restriction is imposed on how many vertices outside the set may be assigned to the same dominator. In many situations, however, this is unrealistic: if the selected vertices represent facilities, service providers, or communication nodes, for example, then each of them can typically handle only a limited number of external users. This leads to the requirement that every chosen vertex have at most $k$ neighbors outside the dominating set. In this way, $k$-limited domination preserves the covering aspect of classical domination while adding a local restriction on the number of external vertices assigned to each dominating vertex.

Formally, as introduced in \cite{podgorica1}, let $G$ be a graph and let
$k\ge 1$ be an integer. A set $D\subseteq V(G)$ is called a
\emph{$k$-limited dominating set} if $D$ is a dominating set of $G$ and
$$
|N_G(u)\setminus D|\le k
\qquad\text{for every }u\in D.
$$
The minimum cardinality among all $k$-limited dominating sets of $G$ is the
\emph{$k$-limited domination number} of $G$, denoted by $\gammakL(G)$.
We will refer to the above local degree condition as the
\emph{$k$-limited constraint}. If $k\ge \Delta(G)$, then every dominating
set is automatically $k$-limited, and hence
$\gammakL(G)=\gamma(G)$. Thus, for a fixed graph $G$, the nontrivial
range is $1\le k<\Delta(G)$.

The initial study in \cite{podgorica1} laid the foundations for the systematic investigation of this parameter. It established general bounds on $\gammakL(G)$ and determined exact values for several basic graph families. It also showed that, for \emph{efficient graphs}, that is, graphs admitting a dominating set in which every vertex is dominated exactly once, the parameter coincides with the classical domination number once $k$ is sufficiently large, and that under suitable structural assumptions one can derive improved upper bounds in terms of $\gamma(G)$.
In addition, strong links with other domination-type parameters were uncovered. In particular, a close connection with \emph{$(1,t)$-domination}  was established, where every vertex outside the dominating set must have at least one neighbor in the set, while every vertex inside the set must have at least $t$ neighbors in the set \cite{fakhran,favaron}. For connected $d$-regular graphs and $1\le k<d$, this relation takes the precise form
$\gammakL(G)=\gamma_{1,d-k}(G)$.
Finally, the case $k=1$ was studied in more detail via its connection with the packing number, leading to further structural results and characterizations. 

The present paper continues this line of research in two directions. First, we investigate the algorithmic complexity of the parameter and prove in Section~\ref{sec:NP} that \textsc{$k$-Limited Dominating Set} is $\mathsf{NP}$-complete for every fixed integer $k\ge 2$. Second, we initiate a systematic study of $k$-limited domination in Cartesian products. Before turning to these two topics, in Section~\ref{prelim} we collect the notation and auxiliary results needed later. The Cartesian product aspect is treated in Section~\ref{car}, where we establish general bounds for $\gammakL(G\square H)$, show that both are sharp, and derive exact values for several natural families of products, including rook graphs, grid graphs, and selected cases involving prisms and hypercubes. We conclude in Section~\ref{sec:conclusion} with several open problems suggested by our results, particularly in the case of hypercubes.

\section{Preliminaries}
\label{prelim}

Let $G$ be a finite simple graph. Its vertex set and edge set are denoted by $V(G)$ and $E(G)$, respectively, and its order by $n(G)=|V(G)|$. For a vertex $v\in V(G)$, we write $N_G(v)$ for the open neighborhood of $v$ and $\deg_G(v)$ for its degree. The maximum degree of $G$ is denoted by $\Delta(G)$. When the graph is clear from the context, we omit the subscript. Also, for ease of notation, we write $[k]=\{1,\dots,k\}$ for every positive integer $k$.

A set $D\subseteq V(G)$ is a \emph{dominating set} of $G$ if every vertex of $V(G)\setminus D$ has a neighbor in $D$. The minimum cardinality among all dominating sets of $G$ is the \emph{domination number} of $G$, denoted by $\gamma(G)$. Let $k\ge 1$ be an integer. A dominating set $D\subseteq V(G)$ is called a \emph{$k$-limited dominating set} if
$|N_G(u)\setminus D|\le k$ for every $u\in D$.
The minimum cardinality among all $k$-limited dominating sets of $G$ is the \emph{$k$-limited domination number} of $G$, denoted by $\gammakL(G)$, \cite{podgorica1}.

If $k\ge \Delta(G)$, then every dominating set of $G$ is automatically
$k$-limited, and hence $\gammakL(G)=\gamma(G)$. Moreover, if
$1\le k<\ell$, then every $k$-limited dominating set is also an
$\ell$-limited dominating set, and therefore
$\gamma_\ell^{\mathrm L}(G)\le \gamma_k^{\mathrm L}(G)$.
Thus, for a fixed graph $G$, the nontrivial range is
$1\le k<\Delta(G)$.

The following general bounds for $k$-limited domination were proved in \cite{podgorica1}.

\begin{theorem} 
\cite{podgorica1}
\label{bounds}
Let $G$ be a connected graph of order $n\geq 3$, and $1\leq k < \Delta(G)$. Then
$$
\max\left\{\gamma(G),\ \left\lceil \frac{n}{k+1} \right\rceil\right\} 
\leq \gammakL(G)  \leq 
n-k.
$$
\end{theorem}

Several graph families were also considered in \cite{podgorica1}. In what follows, we will use the $k$-limited domination numbers of paths and complete graphs.

\begin{proposition} \cite{podgorica1}
\label{families}
Let $k$ be a positive integer. Then the following holds.
\begin{enumerate}[label=(\roman*)]
\item For every $n\geq 3$ we have $\gammaenaL(P_n) = \lceil\frac n2\rceil$.
\item\label{polni} For every $n\geq 3$ and $k$ such that $1\leq k<n-1$ we have $\gammakL(K_n) = n - k$.
\end{enumerate}
\end{proposition}

As already mentioned, one of the main goals of this paper is to initiate a systematic study of $k$-limited domination in Cartesian products. We therefore conclude this section by recalling the following definition.
The \textit{Cartesian product} $G \square H$ has vertex set
$V(G \square H) = V(G) \times V(H)$, where two vertices
$(g,h)$ and $(g',h')$ are adjacent if and only if either
$g=g'$ and $hh' \in E(H)$, or $h=h'$ and $gg' \in E(G)$. 

For later use, we also introduce the following terminology. Let $G$ and $H$ be graphs with
$V(G)=\{g_1,\ldots,g_m\}$ and $V(H)=\{h_1,\ldots,h_n\}$.
For $j\in[n]$, we call
$R_j=\{(g_i,h_j):i\in[m]\}$
the $j$th \emph{row} of $G\square H$, and for $i\in[m]$, we call
$C_i=\{(g_i,h_j):j\in[n]\}$
the $i$th \emph{column}.

%%%%%%%%%%%%%%%%%%%%%
\section{Computational Complexity}
\label{sec:NP}

In this section we investigate the algorithmic side of $k$-limited domination.
Although the notion is closely related to classical domination, the extra local restriction on the dominating vertices does not lead to an easier decision problem.
As announced in the introduction, we prove that the \textsc{$k$-Limited Dominating Set} problem, whose decision version is given in Table~\ref{tabela}, is $\mathsf{NP}$-complete for every fixed integer $k\ge 2$.
Throughout this section we work with an arbitrarily chosen integer $k\ge 2$, while $\ell$ is part of the input.

\begin{table}[h!]
\centering
\begin{tabular}{ |l| }
\hline
\textsc{Instance:} A graph $G$ and a positive integer $\ell$.\\
\hline
\textsc{Question:} Does $G$ admit a $k$-limited dominating set $D$ with $|D|\le \ell$?\\
\hline
\end{tabular}
\caption{\textsc{$k$-Limited Dominating Set}}
\label{tabela}
\end{table}

\smallskip
Given a set $D\subseteq V(G)$, one can check in polynomial time whether $D$ dominates $G$ and whether every vertex
$u\in D$ satisfies the $k$-limited constraint $|N_G(u)\setminus D|\le k$.
Hence, \textsc{$k$-Limited Dominating set} belongs to $\mathsf{NP}$.

To prove $\mathsf{NP}$-completeness, we give a polynomial reduction from the classical \textsc{Dominating Set} problem,
formulated in Table~\ref{tabela2}, which is known to be $\mathsf{NP}$-complete~\cite{garey}.

\begin{table}[h!]
\centering
\begin{tabular}{ |l| }
\hline
\textsc{Instance:} A graph $G$ and a positive integer $\ell$.\\
\hline
\textsc{Question:} Does $G$ have a dominating set $D$ with $|D|\le \ell$?\\
\hline
\end{tabular}
\caption{\textsc{Dominating Set}}
\label{tabela2}
\end{table}

Let $G$ be an arbitrary graph of order $n$ with vertex set $V(G)=\{v_1,\dots,v_n\}$ and $m=|E(G)|$.
From $G$ we construct a graph $G'$ by replacing each edge with a constant-size gadget $S$ that enforces a fixed minimum
contribution in any $k$-limited dominating set.

The gadget $S$ is obtained from the complete graph $K_k$ on vertex set $\{c,d_1,\dots,d_{k-1}\}$ by adding two new vertices
$a$ and $b$ together with the edges $ab$ and $bc$.
Furthermore, for each $t\in[k]$ we add two new vertices $x_t$ and $y_t$ and the edges $bx_t$ and $x_ty_t$.
No other edges are added.
In particular, $|V(S)|=3k+2$. Observe that the encircled subgraph $S^{ij}$ in Figure~\ref{Sgraf} is isomorphic to $S$.

We now define $G'$.
For each edge $v_iv_j\in E(G)$ we subdivide $v_iv_j$ once by a new vertex $c^{ij}$, that is, we delete $v_iv_j$ and add the edges
$v_ic^{ij}$ and $c^{ij}v_j$.
Next, for each edge $v_iv_j\in E(G)$ we take a copy $S^{ij}$ of $S$ and identify its vertex corresponding to $c$ with the subdivision vertex $c^{ij}$.
We denote the remaining vertices of $S^{ij}$ by $a^{ij}$, $b^{ij}$, $d^{ij}_1,\dots,d^{ij}_{k-1}$ and $x^{ij}_t,y^{ij}_t$ for $t\in[k]$, see Figure~\ref{Sgraf}.
Denote the resulting graph by $G'$.

Since the construction only replaces edges of $G$, each vertex of $G$ naturally remains a vertex of $G'$.
Thus we regard $V(G)$ as a subset of $V(G')$, and refer to its elements as the \emph{original vertices} of $G'$.
Clearly, $|V(G')|=n+(3k+2)m$, and the construction can be carried out in polynomial time.

%%%%%%%%%%%%%%%%%%%%%
\begin{figure}[ht!]
\begin{center}
\begin{tikzpicture}[scale=0.6]

\node [My Style, name=c, label=above left:$c^{ij}$ ]   at (0,0) {};		
\node [My Style, name=b, label=above left:$b^{ij}$]   at (0,2) {};	
\node [My Style, name=a, label=above left:$a^{ij}$]    at (0,4) {};	
\node [My Style, name=x1, label=above:$x^{ij}_1$]   at (1.5,2.8) {};	
\node [My Style, name=xk, label=below:$x^{ij}_k$ ]   at (1.5,1.2) {};	
\node [My Style, name=y1, label=above:$y^{ij}_1$]   at (2.8,2.8) {};	
\node [My Style, name=yk, label=below:$y^{ij}_k$]    at (2.8,1.2) {};	
\node [My Style, name=vi, label=above:$v_i$]   at (-5,0) {};	
\node [My Style, name=vj, label=above:$v_j$]   at (5,0) {};
\node at (2.15,2.1){$\vdots$};

\node [My Style, name=d1, label=below:$d^{ij}_1$ ]   at (-1.5,-1.5) {};	
\node [My Style, name=d2, label=below:$d^{ij}_2$ ]   at (-0.5,-1.5) {};	
\node [My Style, name=dk, label=below:$d^{ij}_{k-1}$ ]   at (1.5,-1.5) {};	
\node at (0.5,-1.5){$\dots$};
		
\draw[thick] (vi)--(c)--(vj);
\draw[thick] (c)--(b)--(a);
\draw[thick] (b)--(x1)--(y1);
\draw[thick] (b)--(xk)--(yk);
\draw[thick] (c)--(d1);
\draw[thick] (c)--(d2);
\draw[thick] (c)--(dk);

\draw[thick] (d1)--(d2);
\draw[thick, bend left] (d1) to (dk);
\draw[thick, bend left] (d2) to (dk);

\draw[rounded corners=10pt, thick] (-3,-3.2) rectangle (3.8,5);
%\draw[rounded corners=10pt, thick] (-2.5,-3) rectangle (2.5,-0.5);
\node[font=\large\bfseries] at (0.4,5.7){$S^{ij}$};

%\node [My Style, name=vio, label=above:$v_i$]   at (-11,0) {};
%\node [My Style, name=vjo, label=above:$v_j$]   at (-9,0) {};
%\draw[thick] (vio)--(vjo);

%\node[font=\huge] at (-7,0){$\longrightarrow$};

\end{tikzpicture}
\end{center}
\caption{The gadget $S$ placed between vertices $v_i$ and $v_j$ of  $G$.} 
\label{Sgraf}
\end{figure}
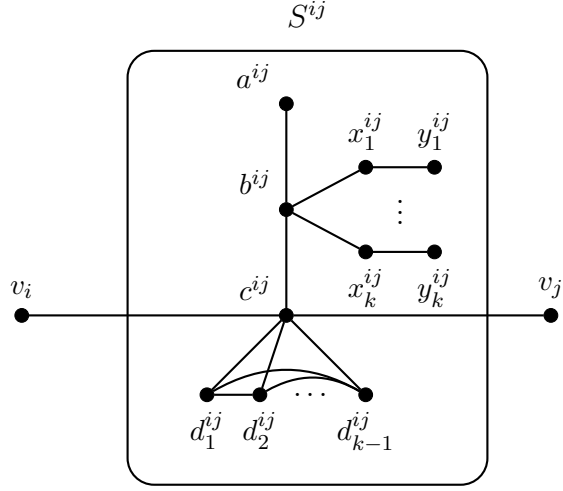
%%%%%%%%%%%%%%%

The next proposition states the correctness of the reduction.

\begin{proposition}
\label{NP:k>1}
Let $k\ge 2$ be a fixed integer, and let $G'$ be constructed from $G$ as described above.
Then $G$ has a dominating set of size at most $\ell$ if and only if $G'$ has a $k$-limited dominating set
of size at most $\ell+(k+2)m$.
\end{proposition}

\begin{proof}
Assume that $G$ has a dominating set $D\subseteq V(G)$ with $|D|\le \ell$.
We define a set $D'\subseteq V(G')$ by taking $D$ together with a fixed selection of vertices from each gadget.
More precisely, for every edge $v_iv_j\in E(G)$ we add to $D'$ the vertices $b^{ij}$ and $x^{ij}_t$ for all $t\in[k]$.
In addition, we include $c^{ij}$ whenever $v_i\in D$ or $v_j\in D$, and otherwise we include $d^{ij}_1$.
Formally,
$$D' \;=\; D \;\cup\!\!\!\bigcup_{v_iv_j\in E(G)}\!\!\!\Bigl(\{b^{ij}\}\cup\{x^{ij}_t:t\in[k]\}\cup Z^{ij}\Bigr),$$
where $Z^{ij}=\{c^{ij}\}$ if $v_i\in D$ or $v_j\in D$, and $Z^{ij}=\{d^{ij}_1\}$ otherwise.
Each gadget contributes exactly $k+2$ vertices, and therefore
$$|D'|=|D|+(k+2)m\le \ell+(k+2)m.$$
We claim that $D'$ is a $k$-limited dominating set of $G'$.

Fix an edge $v_iv_j\in E(G)$ and consider the corresponding copy $S^{ij}$.
For every $t\in[k]$, the vertex $y^{ij}_t$ is adjacent to $x^{ij}_t\in D'$, and $a^{ij}$ is adjacent to $b^{ij}\in D'$.
Moreover, since $\{c^{ij},d^{ij}_1,\dots,d^{ij}_{k-1}\}$ induces a clique and we ensured that at least one of $c^{ij}$ and $d^{ij}_1$
belongs to $D'$, all other vertices in this clique are dominated as well.
Thus every vertex inside every gadget is dominated. Now let $v_i$ be an original vertex of $G'$ with $v_i\notin D'$. Since $D\subseteq D'$, the assumption $v_i\notin D'$ immediately implies $v_i\notin D$. Since $D$ is a dominating set of $G$, there exists a neighbor $v_j\in D$ such that $v_iv_j\in E(G)$.
Consider the subdivision vertex $c^{ij}$ corresponding to the edge $v_iv_j$.
By the definition of $D'$,  $v_j\in D$ implies that $c^{ij}\in D'$.
As $v_i$ is adjacent to $c^{ij}$ in $G'$, it follows that $v_i$ is dominated in $G'$.
This shows that $D'$ is a dominating set of $G'$.

It remains to verify the $k$-limited constraint. Let $u\in D'$. If $u$ is an original vertex $v_i\in D$, then all its neighbors in $G'$ are subdivision vertices $c^{ij}$, and each such $c^{ij}$ is included in $D'$ by construction, hence $|N_{G'}(u)\setminus D'|=0$.
If $u\in D'\cap V(S^{ij})$ is a gadget vertex different from $c^{ij}$, then $u\in\{x^{ij}_t,b^{ij},d^{ij}_1\}$.
In the first two cases, $u$ has at most two neighbors outside $D'$. In the third case, when $u=d^{ij}_1$, we have $\deg_{G'}(u)=k-1$.
Thus $|N_{G'}(u)\setminus D'|\le k$ holds for all such vertices. Finally, consider a subdivision vertex $c^{ij}\in D'$. It has degree $k+2$, and it is adjacent to $b^{ij}\in D'$ and to at least one of $v_i,v_j\in D\subseteq D'$ by the definition of $D'$.
Therefore $c^{ij}$ has at most $k$ neighbors outside $D'$.
This completes the proof that $D'$ is a $k$-limited dominating set of size at most $\ell+(k+2)m$.

\medskip

Conversely, let $D'$ be a $k$-limited dominating set of $G'$ with $|D'|\le \ell+(k+2)m$.
Define
$D_0= D'\cap V(G)$, that is, the set of original vertices of $G$
that belong to $D'$.

We first show that gadgets are designed so that each of them forces at least $k+2$ vertices in any $k$-limited dominating set of $G'$. 
Fix $v_iv_j\in E(G)$ and consider $S^{ij}$.
For each $t\in[k]$, the vertex $y^{ij}_t$ has the unique neighbor $x^{ij}_t$, hence
$\{x^{ij}_t,y^{ij}_t\}\cap D'\neq\emptyset$, contributing at least $k$ vertices from $V(S^{ij})$.
Also, $a^{ij}$ has the unique neighbor $b^{ij}$, so $\{a^{ij},b^{ij}\}\cap D'\neq\emptyset$, contributing one more vertex.
Finally, every $d^{ij}_s$, where $s\in [k-1]$, has neighbors only within the clique on $\{c^{ij},d^{ij}_1,\dots,d^{ij}_{k-1}\}$, so domination of
$\{d^{ij}_1,\dots,d^{ij}_{k-1}\}$ forces $D'$ to contain at least one vertex from this clique.
Therefore, for every edge $v_iv_j\in E(G)$, the corresponding copy $S^{ij}$ satisfies
$|D'\cap V(S^{ij})|\ge k+2$.

Applying this bound to each of the $m$ gadgets and summing the resulting inequalities yields
$$|D_0|
=|D'|-\sum_{v_iv_j\in E(G)}|D'\cap V(S^{ij})|
\le |D'|-(k+2)m
\le \ell.$$
Thus, $|D_0|\le \ell$. If $D_0$, seen as a subset of $G$, already dominates $G$, then we are done.

Assume therefore that $D_0$ is not a dominating set of $G$.
We will show that $D'$ can be modified iteratively into a $k$-limited dominating set $\widehat{D}$ of $G'$ with $|\widehat{D}|\le |D'|$ such that the set of original vertices of $G'$ contained in $\widehat{D}$ forms a dominating set of $G$.
Let
$$W=\{\,v\in V(G)\setminus D_0:\ N_G(v)\cap D_0=\emptyset\,\}$$
be the set of vertices not dominated by $D_0$ in $G$.
Since $D_0$ is not a dominating set of $G$, the set $W$ is nonempty.
Choose a vertex $v_i\in W$.

As $D'$ dominates $G'$, the original vertex $v_i$ has a neighbor in $D'$.
By construction of $G'$, the neighbors of $v_i$ in $G'$ are precisely the subdivision vertices $c^{ir}$ corresponding to edges $v_iv_r\in E(G)$.
Hence there exists a neighbor $v_j\in N_G(v_i)$ such that $c^{ij}\in D'$.
Since $v_i\in W$, none of its neighbors in $G$ belongs to $D_0$.
In particular, the vertex $v_j$ chosen above does not belong to $D_0$.
As $v_j$ is an original vertex and $D_0=D'\cap V(G)$, it follows that $v_j\notin D'$.

We already know that $|D'\cap V(S^{ij})|\ge k+2$.
We now show that under the additional assumptions $c^{ij}\in D'$ and $v_i,v_j\notin D'$, one more gadget vertex is forced to be in $D'$, that is,
$|D'\cap V(S^{ij})|\ge k+3$.
For each $t\in[k]$, the leaf $y^{ij}_t$ must be dominated by $x^{ij}_t$ or by itself, so
$\{x^{ij}_t,y^{ij}_t\}\cap D'\neq\emptyset$, yielding at least $k$ vertices of $D'$ in $V(S^{ij})$.
Moreover, since $a^{ij}$ is adjacent only to $b^{ij}$, we have $\{a^{ij},b^{ij}\}\cap D'\neq\emptyset$.
It remains to show that $D'$ contains at least one vertex among $d^{ij}_1,\dots,d^{ij}_{k-1}$.
Recall that $c^{ij}\in D'$ and $v_i,v_j\notin D'$. Since $\deg_{G'}(c^{ij})=k+2$, the vertex $c^{ij}$ already has two neighbors outside $D'$,
namely $v_i$ and $v_j$. As $D'$ is $k$-limited, at least two vertices from
$\{b^{ij},d^{ij}_1,\dots,d^{ij}_{k-1}\}$ must belong to $D'$, which further implies that at least one of $d^{ij}_1,\dots,d^{ij}_{k-1}$ belongs to $D'$. Therefore
$|D'\cap V(S^{ij})|\ge k+3$.

We now adjust $D'$ locally around $v_i$ to obtain another $k$-limited dominating set $D''$ in which $v_i$ is selected.
For each edge $v_iv_r\in E(G)$, let
$$C^{ir}=\{c^{ir},\, b^{ir}\}\cup\{x^{ir}_t:\ t\in[k]\}.$$
Then $|C^{ir}|=k+2$, and $C^{ir}$ dominates $S^{ir}$, see Figure~\ref{Sir}. 
Now define
$$D'' = \Bigl(D'\setminus \bigcup_{v_iv_r\in E(G)} (D'\cap V(S^{ir}))\Bigr)\ \cup\ \{v_i\}\ \cup\ \bigcup_{v_iv_r\in E(G)} C^{ir}.$$
In other words, we replace the part of $D'$ inside each gadget incident with $v_i$ by the corresponding set $C^{ir}$, and then add $v_i$.

%%%%%%%%%%%%%%%%%%%%%
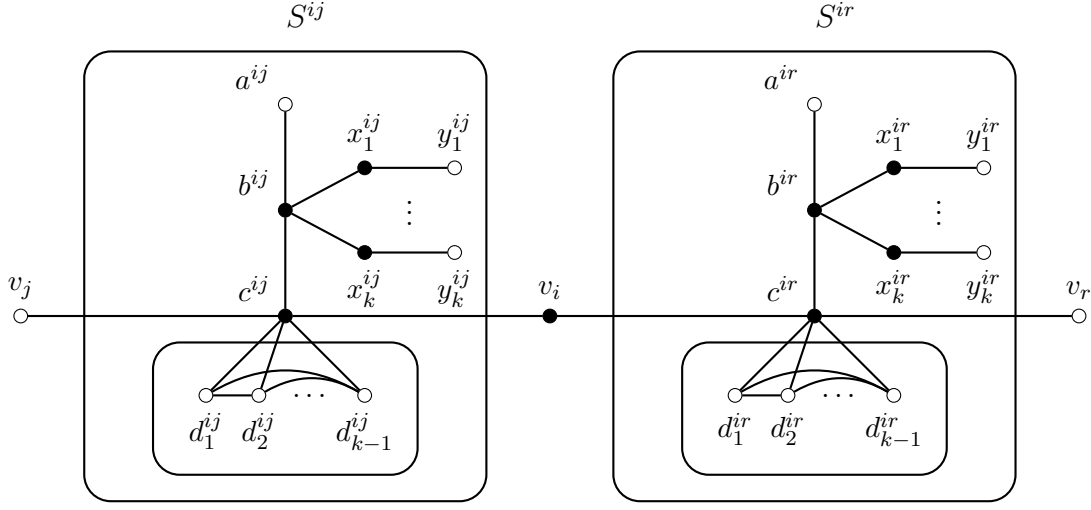
\begin{figure}[ht!]
\begin{center}
\begin{tikzpicture}[scale=0.62]

%%% LEVI $$$			
\node [My Style, name=c, label=above left:$c^{ij}$ ]   at (0,0) {};		
\node [My Style, name=b, label=above left:$b^{ij}$]    at (0,2) {};	
\node [My Style2, name=a, label=above left:$a^{ij}$]   at (0,4) {};	
\node [My Style, name=x1, label=above:$x^{ij}_1$]      at (1.5,2.8) {};	
\node [My Style, name=xk, label=below:$x^{ij}_k$ ]     at (1.5,1.2) {};	
\node [My Style2, name=y1, label=above:$y^{ij}_1$]     at (3.2,2.8) {};	
\node [My Style2, name=yk, label=below:$y^{ij}_k$]     at (3.2,1.2) {};	
\node [My Style2, name=vj, label=above:$v_j$]          at (-5,0) {};	
\node [My Style,  name=vi, label=above:$v_i$]          at (5,0) {};
\node at (2.35,2.1){$\vdots$};

\node [My Style2, name=d1, label=below:$d^{ij}_1$ ]    at (-1.5,-1.5) {};	
\node [My Style2, name=d2, label=below:$d^{ij}_2$ ]    at (-0.5,-1.5) {};	
\node [My Style2, name=dk, label=below:$d^{ij}_{k-1}$ ] at (1.5,-1.5) {};	
\node at (0.5,-1.5){$\dots$};
		
\draw[thick] (vi)--(c)--(vj);
\draw[thick] (c)--(b)--(a);
\draw[thick] (b)--(x1)--(y1);
\draw[thick] (b)--(xk)--(yk);
\draw[thick] (c)--(d1);
\draw[thick] (c)--(d2);
\draw[thick] (c)--(dk);

\draw[thick] (d1)--(d2);
\draw[thick, bend left] (d1) to (dk);
\draw[thick, bend left] (d2) to (dk);

\draw[rounded corners=10pt, thick] (-3.8,-3.5) rectangle (3.8,5);
\draw[rounded corners=10pt, thick] (-2.5,-3) rectangle (2.5,-0.5);
\node[font=\large\bfseries] at (0.4,5.7){$S^{ij}$};

%%% DESNI $$$
\node [My Style, name=cp, label=above left:$c^{ir}$ ]   at (10,0) {};		
\node [My Style, name=bp, label=above left:$b^{ir}$]    at (10,2) {};	
\node [My Style2, name=ap, label=above left:$a^{ir}$]   at (10,4) {};	
\node [My Style, name=x1p, label=above:$x^{ir}_1$]      at (11.5,2.8) {};	
\node [My Style, name=xkp, label=below:$x^{ir}_k$ ]     at (11.5,1.2) {};	
\node [My Style2, name=y1p, label=above:$y^{ir}_1$]     at (13.2,2.8) {};	
\node [My Style2, name=ykp, label=below:$y^{ir}_k$]     at (13.2,1.2) {};
\node [My Style2, name=vr, label=above:$v_r$]           at (15,0) {};
\node at (12.35,2.1){$\vdots$};

\node [My Style2, name=d1p, label=below:$d^{ir}_1$ ]    at (8.5,-1.5) {};	
\node [My Style2, name=d2p, label=below:$d^{ir}_2$ ]    at (9.5,-1.5) {};	
\node [My Style2, name=dkp, label=below:$d^{ir}_{k-1}$ ] at (11.5,-1.5) {};	
\node at (10.5,-1.5){$\dots$};
		
\draw[thick] (vi)--(cp)--(vr);
\draw[thick] (cp)--(bp)--(ap);
\draw[thick] (bp)--(x1p)--(y1p);
\draw[thick] (bp)--(xkp)--(ykp);
\draw[thick] (cp)--(d1p);
\draw[thick] (cp)--(d2p);
\draw[thick] (cp)--(dkp);

\draw[thick] (d1p)--(d2p);
\draw[thick, bend left] (d1p) to (dkp);
\draw[thick, bend left] (d2p) to (dkp);

\draw[rounded corners=10pt, thick] (6.2,-3.5) rectangle (13.8,5);
\draw[rounded corners=10pt, thick] (7.5,-3) rectangle (12.5,-0.5);
\node[font=\large\bfseries] at (10.4,5.7){$S^{ir}$};

\end{tikzpicture}
\end{center}
\caption{A local view of the modification: black vertices belong to $D''$.}
\label{Sir}
\end{figure}

We first show that $|D''|\le |D'|$.
For every neighbor $v_r$ of $v_i$, the lower bound proved above gives
$|D'\cap V(S^{ir})|\ge k+2=|C^{ir}|$,
so replacing $D'\cap V(S^{ir})$ by $C^{ir}$ does not increase the size.
Moreover, for the specific neighbor $v_j$ chosen above we have
$|D'\cap V(S^{ij})|\ge k+3,$
and hence on the gadget $S^{ij}$ the replacement decreases the size by at least $1$.
This compensates for the addition of $v_i$, and therefore
$|D''|\le |D'|$.

Next, we verify that $D''$ is still a $k$-limited dominating set of $G'$.
First, domination is preserved.
Each modified gadget $S^{ir}$ is dominated by $C^{ir}$, and since $c^{ir}\in C^{ir}$, both original vertices $v_i$ and $v_r$ remain dominated.
All other vertices of $V(G')\setminus D''$ lie outside the modified gadgets, and their neighborhoods are unchanged.
Hence these vertices remain dominated as well.
It remains to check that the $k$-limited constraint holds for every $u\in D''$.
The newly added original vertex $v_i$ has no neighbors outside $D''$, since all its neighbors in $G'$ are subdivision vertices $c^{ir}$ and all of them belong to $D''$ by construction.
Now let $u\in D''\cap V(S^{ir})$ for some edge $v_iv_r\in E(G)$.
If $u=c^{ir}$, then $\deg_{G'}(u)=k+2$ and both $b^{ir}$ and $v_i$ belong to $N_{G'}(u)\cap D''$, so $|N_{G'}(u)\setminus D''|\le k$.
For $u\in \{b^{ir}, x^{ir}_1, \ldots, x^{ir}_k\}$ we have $|N_{G'}(u)\setminus D''|=1< k$. For all remaining vertices of $D''$, the number of neighbors outside the set remains the same, so the $k$-limited constraint continues to hold, and $D''$ is again a $k$-limited dominating set of $G'$.

Finally, the modification ensures that the original vertex $v_i$ is now selected in $D''$, since
$D''\cap V(G)=D_0\cup\{v_i\}$.
Consequently, $v_i$ is no longer undominated in $G$.
%Moreover, we never remove original vertices from the set, the only change on $V(G)$ is the addition of $v_i$.
%Therefore the set of vertices not dominated in $G$ strictly decreases.

If the set of original vertices  contained in $D''$ does not yet dominate $G$, we repeat the same process.
Since at each step the set of undominated vertices in $G$ is strictly reduced and $V(G)$ is finite, the procedure terminates after finitely many steps.
Thus we obtain a $k$-limited dominating set $\widehat{D}$ of $G'$ with
$|\widehat{D}|\le |D'|$
such that the set of original vertices contained in $\widehat{D}$ forms a dominating set of $G$.

It remains to bound its size.
Since $\widehat{D}$ is again a $k$-limited dominating set of $G'$, the argument above applies to $\widehat{D}$ as well: every gadget contributes at least $k+2$ vertices to $\widehat{D}$.
Hence
$$|\widehat{D}\cap V(G)|
=|\widehat{D}|-\sum_{v_iv_j\in E(G)} |\widehat{D}\cap V(S^{ij})|
\le |\widehat{D}|-(k+2)m
\le |D'|-(k+2)m
\le \ell.$$
Therefore $\widehat{D}\cap V(G)$ is a dominating set of $G$ of size at most $\ell$.
This completes the proof.
\end{proof}

Proposition~\ref{NP:k>1} establishes the correctness of the reduction. Since the reduction is polynomial and \textsc{$k$-Limited Dominating set} belongs to $\mathsf{NP}$, we obtain the following.

\begin{theorem}
\label{prop:NPk}
For every fixed integer $k\ge 2$, the \textsc{$k$-Limited Dominating set} problem is $\mathsf{NP}$-complete.
\end{theorem}

In the case $k=2$, our reduction produces bipartite instances.
Indeed, subdividing every edge of $G$ once yields a bipartite graph, and the gadget $S$ for $k=2$ is bipartite as well.
Hence the graph $G'$ constructed by the reduction is bipartite.
Therefore, the same reduction shows that the hardness result extends to bipartite graphs.
Consequently, we obtain the following.

\begin{corollary}\label{cor:NP-bip-k2}
\textsc{$2$-Limited Dominating set} problem is $\mathsf{NP}$-complete even when restricted to bipartite graphs.
\end{corollary}

%%%%%%%%%%%%%%%%%%%%%%%%%%
\section{Cartesian product}
\label{car}

We next investigate how $k$-limited domination behaves under the Cartesian product of graphs.
We begin with a family of rook graphs that will later show that the general upper bound for Cartesian products is best possible.

\begin{proposition}
\label{rook}
Let $k$ be a positive integer. Then $\gammakL(K_{k+2} \square K_{k+2})=2(k+2)$. 
\end{proposition}

\begin{proof}
Let $m=k+2$. Take two rows, say $R_1$ and $R_2$, and let $D=R_1\cup R_2$. Then $D$ dominates $K_m\square K_m$, and every vertex of $D$ has exactly $m-2=k$ neighbors outside $D$. Hence $D$ is a $k$-limited dominating set, and therefore
$\gamma_k^{\mathrm L}(K_m\square K_m)\le 2m$.

For the reverse inequality, let $D$ be a
$\gamma_k^{\mathrm L}(K_m\square K_m)$-set. For each $j\in[m]$, let
$r_j=|D\cap R_j|$,
and for each $i\in[m]$, let
$c_i=|D\cap C_i|$.
Then
$$\sum_{j=1}^m r_j = |D| = \sum_{i=1}^m c_i.$$
We first note that no row and no column contains exactly one vertex of $D$, as the opposite case contradicts the $k$-limited constraint.
Thus $r_j \ne 1$ for all $j\in [m]$ and $c_i \ne 1$ for all $i\in[m]$.
Now suppose, to the contrary, that $|D|\le 2m-1$. Then
$$\frac{1}{m} \sum_{j=1}^m r_j = \frac{|D|}{m}
\le \frac{2m-1}{m} < 2,$$
hence there exists a row $R_{j_0}$ with $r_{j_0} \le 1$.
Since $r_{j_0} \ne 1$, we must have $r_{j_0} = 0$, that is, the
row $R_{j_0}$ contains no vertex in $D$.
Applying the same argument to the sum $\sum_{i=1}^m c_i$ shows that there exists a column $C_{i_0}$ with
no vertex in $D$.
But then the vertex $(i_0,j_0)$ has no neighbor in $D$, contradicting the
fact that $D$ is a dominating set of $K_m \square K_m$.

Therefore $|D|\ge 2m$. Combining this inequality with the upper bound
derived above concludes the proof.
\end{proof}

We now turn to general Cartesian products. The following theorem gives universal lower and upper bounds and shows that both are tight. 

\begin{theorem}
\label{CP-bounds}
Let $G$ and $H$ be graphs, and let $k\geq 1$ be an integer.
Then
$$
\left\lceil \frac{n(G)\, n(H)}{k+1} \right\rceil
\le
\gammakL(G \square H)
\le
\min\big\{\gammakL(G)\,n(H),\, \gammakL(H)\,n(G)\big\},
$$
and both bounds are sharp for every $k$.
\end{theorem}

\begin{proof}
Let $D$ be a $k$-limited dominating set of $G \square H$.
For each vertex $x\in D$, the set $\{x\}\cup \bigl(N_{G\square H}(x)\setminus D\bigr)$
has size at most $k+1$, since $|N_{G\square H}(x)\setminus D|\le k$ by the $k$-limited constraint.
Moreover, because $D$ dominates $G\square H$, every vertex of $G\square H$ belongs either to $D$ or to $N_{G\square H}(x)\setminus D$ for some $x\in D$.
Hence
$$
V(G\square H)\subseteq \bigcup_{x\in D}\Bigl(\{x\}\cup \bigl(N_{G\square H}(x)\setminus D\bigr)\Bigr),
$$
and therefore
$$
n(G)\, n(H)=|V(G\square H)|
\leq
\sum_{x\in D}(k+1)
=
(k+1)|D|,
$$
which yields the lower bound.

For the upper bound, let $D_G \subseteq V(G)$ be a $\gamma^{\mathrm L}_k(G)$-set.
Let
$D = D_G \times V(H) \subseteq V(G \square H)$.
We claim that $D$ is a $k$-limited dominating set of $G \square H$.

Let $(u,v)\in V(G \square H)\setminus D$ be  arbitrary. Then $u \notin D_G$. As $D_G$ is, in particular, a dominating set of $G$, there exists a
vertex $x \in D_G$ with $ux \in E(G)$. Then $(u,v)$ and $(x,v)$ are adjacent
in $G \square H$, and $(x,v) \in D$. Hence $D$ dominates $G \square H$. 

To verify the $k$-limited constraint now let $(u,v)$ be an arbitrary vertex in $D$. Then $u \in D_G$ and $v \in V(H)$.
There are two types of neighbors of $(u,v)$ in $G \square H$:
$(u',v)$ with $u'u \in E(G)$,
and $(u,v')$ with $vv' \in E(H)$.
Since $D=D_G\times V(H)$, every neighbor of the form $(u,v')$ also belongs to $D$.
Therefore
$N_{G\square H}((u,v))\setminus D
=
\{(u',v): u'u\in E(G),\ u'\notin D_G\},
$
and consequently
$
|N_{G\square H}((u,v))\setminus D|
=
|N_G(u)\setminus D_G|.
$
Because $D_G$ is a $k$-limited dominating set of $G$, we have
$|N_G(u)\setminus D_G|\le k$ for every $u\in D_G$.
Hence
$|N_{G\square H}((u,v))\setminus D|\le k$ for every $(u,v)\in D$.
Thus $D$ is a $k$-limited dominating set of $G\square H$, and so
$$
\gammakL(G\square H)
\le |D|
=|D_G|\,n(H)
=\gammakL(G)\, n(H).
$$
By symmetry, interchanging the roles of $G$ and $H$ yields
$
\gammakL(G\square H)\le \gammakL(H)\, n(G)$.
Taking the minimum of these two bounds yields the stated upper bound.

By Proposition~\ref{rook}, the upper bound is attained for every $k$.
For the lower bound, let $G$ and $H$ be obtained by attaching $k$
leaves to each vertex of arbitrary graphs $F$ and $F'$, respectively.
Then $V(F)$ and $V(F')$, viewed as subsets of $V(G)$ and $V(H)$,
respectively, are minimum $k$-limited dominating sets, so
$\gammakL(G)=\frac{n(G)}{k+1}$
and
$\gammakL(H)=\frac{n(H)}{k+1}$.
Hence the upper and lower bounds from the theorem coincide, showing that
the lower bound is attained for every $k$.
\end{proof}

%%%%%%%%%%%%

We now determine the $1$-limited domination number of \emph{grid} graphs, i.e., the Cartesian products of two paths.

For positive integers $m$ and $n$, let
$G_{m,n}=P_m\square P_n$,
with vertex set
$V(G_{m,n})=[m]\times[n]$, where
two vertices $(i,j)$ and $(i',j')$ are adjacent in $G_{m,n}$ if and
only if
$|i-i'|+|j-j'|=1$.
A \emph{subgrid} of $G_{m,n}$ is an induced subgraph with vertex set
$I\times J$, where $I\subseteq[m]$ and $J\subseteq[n]$ are intervals
of consecutive integers.

We first describe the structure of the components of the subgraph induced
by the vertices outside a $1$-limited dominating set. 
Let $D$ be a $1$-limited dominating set of $G_{m,n}$, let
$\overline D=V(G_{m,n})\setminus D$, and let $H$ denote the subgraph
of $G_{m,n}$ induced by $\overline D$. We use this notation throughout
the discussion on grid graphs.

\begin{lemma}
\label{grid-components}
Every component $Q$ of $H$ is a subgrid. More precisely, if
$V(Q)=I\times J$,
then
$\min\{|I|,|J|\}\le2$.
\end{lemma}

\begin{proof}
Let $Q$ be a component of $H$. Suppose, to the contrary, that
$Q$ is not a subgrid of $G_{m,n}$. Let $I\times J$ be the vertex set
of a subgrid contained in $Q$ that is maximal with respect to inclusion.
Since $I\times J\neq V(Q)$ and $Q$ is connected, some vertex of
$V(Q)\setminus(I\times J)$ is adjacent to a vertex of $I\times J$.
By symmetry, suppose that
$I=\{a,\ldots,b\}$
and that this vertex is
$(b+1,j)$
for some $j\in J$.

If $|J|=1$, then
$\{a,\ldots,b+1\}\times J$
is the vertex set of a larger subgrid contained in $Q$, contradicting
the maximality of $I\times J$. Hence $|J|\ge2$.
Moreover, not all vertices $(b+1,k)$, $k\in J$, belong to $Q$, since
otherwise
$\{a,\ldots,b+1\}\times J$
would again be the vertex set of a larger subgrid contained in $Q$.
Therefore, there exist $k,\ell\in J$ with $|k-\ell|=1$ such that
$(b+1,k)\in Q$ and 
$(b+1,\ell)\notin Q$.

Since $(b+1,\ell)$ is adjacent to $(b+1,k)\in Q$, it cannot belong to
$\overline{D}$, for otherwise it would belong to the same component $Q$.
Hence $(b+1,\ell)\in D$. But $(b+1,\ell)$ is adjacent to both
$(b+1,k)$ and $(b,\ell)$, which belong to $\overline{D}$, contradicting the $1$-limited constraint. Hence every component of
$H$ is a subgrid.

Now let $Q$ be a subgrid in $H$ with
$V(Q)=I\times J$.
If $|I|\ge3$ and $|J|\ge3$, then $Q$ contains a vertex in
$\overline{D}$,  whose all four neighbors also belong to
$\overline{D}$, contradicting the fact that $D$ is a dominating set.
Therefore
$\min\{|I|,|J|\}\le2$.
\end{proof}

We next establish the parity argument needed for odd grids.
We use the following standard terminology. A \emph{matching} is a set of pairwise nonincident edges.
A vertex is \emph{covered} by a matching if it is incident with an edge
of the matching, and a matching \emph{saturates} a vertex set $S$ if
every vertex of $S$ is covered by the matching.
Here and below, a component is said to meet a row if it contains a vertex
of that row.

\begin{lemma}
\label{grid-matching}
Let $m,n\ge3$. Suppose that no component of $H$ meets all $n$ rows.
Then there exists a matching $M$ between $D$ and $\overline D$ that
saturates $\overline D$ and such that, for every $i\in[m-1]$, the
number of edges of $M$ joining $C_i$ and $C_{i+1}$ is even.
\end{lemma}

\begin{proof}
Since $D$ is a dominating set, every vertex of $\overline{D}$ has at least one
neighbor in $D$. Since $D$ is $1$-limited, every vertex of $D$ has at
most one neighbor in $\overline{D}$. Therefore, if for each vertex of
$\overline{D}$ we choose one adjacent vertex in $D$, then the corresponding
edges automatically form a matching between $D$ and $\overline{D}$ that
saturates $\overline{D}$.
Thus, it remains only to make these choices for the different components
of $H$ so that the required parity condition holds. We will do this in
such a way that each component contributes an even number of matching
edges between every pair of consecutive columns.

Let $Q$ be a component of $H$. By Lemma~\ref{grid-components},
$V(Q)=I\times J$,
where $I\subseteq[m]$ and $J\subseteq[n]$ are intervals of consecutive
integers. Let $a=|I|$ and $b=|J|$. Then
$\min\{a,b\}\le2$.
Since $Q$ does not meet all $n$ rows, we also have $b<n$.
We will repeatedly use the fact that every vertex outside $Q$ that is
adjacent to a vertex of $Q$ belongs to $D$, otherwise it would belong
to the same component of $H$.

Suppose first that $a=1$. Thus $Q$ lies in a single column. Choose a
column adjacent to it, which is possible since $m\ge3$. If $b$ is even,
match every vertex of $Q$ to its neighbor in the chosen column. Then
exactly $b$ matching edges have one endvertex in the column containing
$Q$ and the other in an adjacent column.
If $b$ is odd, then $J\neq[n]$ implies that one of the two endvertices
of $Q$ has a neighbor outside $Q$ in the same column. Match this
endvertex to that neighbor and match the remaining $b-1$ vertices to
their neighbors in the chosen adjacent column. Thus exactly $b-1$
matching edges have one endvertex in the column containing $Q$ and the
other in an adjacent column. In either case, $Q$ contributes an even
number of matching edges between every pair of consecutive columns.

Now let $a=2$, say
$I=\{p,p+1\}$.
Suppose first that
$2\le p\le m-2$.
Then both $C_{p-1}$ and $C_{p+2}$ exist.

If $b$ is even, match the vertices of $Q$ in $C_p$ to their neighbors
in $C_{p-1}$ and those in $C_{p+1}$ to their neighbors in $C_{p+2}$.
If $b$ is odd, write $J=\{r,\ldots,s\}$. Since $J\neq[n]$, either
$r>1$ or $s<n$. Assume, without loss of generality, that $r>1$.
Match $(p,r)$ to $(p,r-1)$ and $(p+1,r)$ to $(p+1,r-1)$.
Match the remaining $b-1$ vertices in $C_p$ to their neighbors in
$C_{p-1}$ and the remaining $b-1$ vertices in $C_{p+1}$ to their
neighbors in $C_{p+2}$.
Thus, in either case, each relevant pair of consecutive columns is
crossed by an even number of matching edges. The case $s<n$ is
analogous.

It remains to consider the cases in which either $a\ge3$, or $a=2$
and $Q$ meets the first or the last column. In all these cases,
$b\le2$. Indeed, if $a\ge3$, this follows from
Lemma~\ref{grid-components}. If $a=2$ and $Q$ meets a boundary column,
then $b\ge3$ would imply that an internal vertex of that column has all
its neighbors in $Q\subseteq\overline D$, contradicting the fact that $D$ is a dominating set.

If $b=1$, match every vertex of $Q$ to a neighbor outside $Q$ in the
same column. Such a neighbor exists since $n\ge3$. Hence no matching
edge crosses a pair of consecutive columns.

Let $b=2$. Then $Q$ cannot meet the first or the last row. Indeed, if
$a\ge3$, then a boundary row of $Q$ contains a vertex that is internal
with respect to $I$, and all its neighbors belong to
$Q\subseteq\overline D$, contradicting $D$ being a dominating set. If $a=2$ and
$Q$ meets a boundary column, a corner vertex of the grid is in $Q$ and has
all its neighbors in $Q$. Thus it is  undominated, a contradiction again.

Thus, if $J=\{r,r+1\}$, then
$2\le r\le n-2$, so both rows $R_{r-1}$ and $R_{r+2}$ exist.
Match all vertices of $Q$ in $R_r$ to their neighbors in $R_{r-1}$,
and all vertices in $R_{r+1}$ to their neighbors in $R_{r+2}$.
Again, no matching edge crosses a pair of consecutive columns.

Combining these choices over all components of $H$ yields the required
matching $M$.
\end{proof}

We can now determine the exact value of the $1$-limited domination
number for all grid graphs.

\begin{theorem}
\label{nxm}
For all integers $1\le m\le n$,
\[
\gamma^{\mathrm L}_1(P_m\square P_n)=
\begin{cases}
\dfrac{mn}{2}, & \text{if $m$ or $n$ is even},\\[2mm]
\dfrac{m(n+1)}{2}, & \text{if $m$ and $n$ are odd}.
\end{cases}
\]
\end{theorem}

\begin{proof}
The upper bound follows from Theorem~\ref{CP-bounds} and
Proposition~\ref{families}. If at least one of $m$ and $n$ is even,
then Theorem~\ref{CP-bounds} also gives the matching lower bound
$\gamma_1^{\mathrm L}(G_{m,n})\ge mn/2$, and the result follows.

It remains to consider the case when both $m$ and $n$ are odd.
If one of them is equal to $1$, then $G_{m,n}$ is a path and the result
follows from Proposition~\ref{families}. By symmetry, we may therefore
assume that
$3\le m\le n$.

Assume first that no component of $H$ meets all $n$ rows.
By Lemma~\ref{grid-matching}, there exists a matching $M$ between
$\overline D$ and $D$ that saturates $\overline D$ and such that the
number of matching edges between every pair of consecutive columns is
even.
For each $i\in[m]$, let $q_i$ denote the number of vertices of $C_i$
covered by $M$. Each matching edge with both endvertices in $C_i$
contributes two to $q_i$, while the matching edges between $C_i$ and
either adjacent column contribute an even number by
Lemma~\ref{grid-matching}. Hence $q_i$ is even for every $i\in[m]$.

Since $|C_i|=n$ is odd, $M$ cannot cover all vertices of $C_i$.
Thus at least one vertex is left uncovered in every column, so $M$
leaves at least $m$ vertices uncovered.
Since $M$ saturates $\overline D$, we have
$|M|=|\overline D|$.
Hence $M$ covers all vertices of $\overline D$ and exactly $|\overline D|$ vertices of $D$. Therefore the number of uncovered
vertices is
$|D|-|\overline D|$,
and consequently
$|D|-|\overline D|\ge m$.
Together with
$|D|+|\overline D|=mn$,
this gives
$|D|\ge \frac{m(n+1)}2$.

Now suppose that some component $Q$ of $H$ meets all $n$ rows.
By Lemma~\ref{grid-components},
$V(Q)=I\times[n]$
with $|I|\le2$. Since $m\ge3$, $Q$ does not meet all $m$ columns.
Moreover, no component of $H$ meets all $m$ columns. Indeed, such a
component would, again by Lemma~\ref{grid-components}, have vertex set
$[m]\times J$ for some nonempty $J\subseteq[n]$, and hence
would intersect $I\times[n]$, which is impossible for two distinct
components of $H$.

We may therefore apply Lemma~\ref{grid-matching} with the roles of rows
and columns interchanged. Thus there exists a matching $M$ between
$D$ and $\overline D$ saturating $\overline D$ such that the number of
matching edges between every pair of consecutive rows is even.
Since each row contains $m$ vertices and $m$ is odd, the same parity
argument shows that at least one vertex is left uncovered in every row.
Hence $M$ leaves at least $n$ vertices uncovered, and therefore
$|D|-|\overline D|\ge n\ge m$.
As above, this yields
$|D|\ge \frac{m(n+1)}2$.

Thus, when $m$ and $n$ are both odd and $m\le n$, we have
$\gamma_1^{\mathrm L}(G_{m,n})
\ge \frac{m(n+1)}2$.
Together with the upper bound from Theorem~\ref{CP-bounds}, this completes the proof.
\end{proof}

We conclude this section with a brief discussion of hypercubes, whose regularity, bipartiteness, and recursive structure make them particularly suitable for studying \(k\)-limited domination.

In \cite{podgorica1} the following relation was established for regular graphs.

\begin{proposition}
\cite{podgorica1}
\label{d-regular}
Let $G$ be a connected $d$-regular graph of order $n\geq 3$ and let $k$ be an integer such that $1\leq k<d$. Then $\gammakL(G)=\gamma_{1,d-k}(G)$.
\end{proposition}

Recall that the \emph{prism} over a graph $G$ is the Cartesian product $G\square K_2$. It was shown in \cite{azar}, and later reproved by different arguments in \cite{GodHen18,at-dmgt}, that for every bipartite graph $G$, $\gamma_t(G\square K_2)=2\gamma(G)$.
Now let $G$ be a connected $d$-regular bipartite graph. Then $G\square K_2$ is $(d+1)$-regular, and Proposition~\ref{d-regular}, applied with $k=d$, gives $\gamma_d^{\mathrm L}(G\square K_2)=\gamma_{1,1}(G\square K_2)=\gamma_t(G\square K_2)$. Hence
$\gamma_d^{\mathrm L}(G\square K_2)=2\gamma(G)$.

Hypercubes fit naturally into this setting, since $Q_{d+1}=Q_d\square K_2$
for every $d\ge1$, while $Q_d$ is $d$-regular and bipartite. Therefore
$\gamma_d^{\mathrm L}(Q_{d+1})=2\gamma(Q_d)$.
For hypercubes, exact values of the classical domination number are known only
for small dimensions (see Table~\ref{tabela3}) and for two distinguished
infinite families. Namely, if $d=2^r-1$, then
$\gamma(Q_d)=2^{\,2^r-r-1}$, while if $d=2^r$, then
$\gamma(Q_d)=2^{\,2^r-r}$, see \cite{HarLiv} and \cite{Wee},
respectively. Consequently,
$\gamma_{d-1}^{\mathrm L}(Q_d)=2^{\,2^r-r}$ when $d=2^r$, and
$\gamma_{d-1}^{\mathrm L}(Q_d)=2^{\,2^r-r+1}$ when $d=2^r+1$.
At the opposite end, the case $k=1$ has a particularly simple exact form.

\begin{proposition}
\label{1Lkock}
For every integer $d\ge 1$, 
$\gammaenaL(Q_d) = 2^{d-1}$.
\end{proposition}

\begin{proof}
Since $Q_d$ has $2^d$ vertices, the lower bound in Theorem~\ref{bounds} gives $
\gammaenaL(Q_d)\ge 2^{d-1}$.

To prove the opposite inequality, recall that for an integer $d\ge 1$, $Q_d$ can be seen as the graph with 
$V(Q_d)=\{(x_1,\dots,x_d): x_i\in\{0,1\}\text{ for all }i\in[d]\}$,
where two vertices are adjacent if and only if they differ in exactly one coordinate. Let $
D=\{(x_1,\dots,x_d)\in\{0,1\}^d:\ x_1=0\}$. Clearly $|D|=2^{d-1}$. We claim that $D$ is a $1$-limited dominating set of $Q_d$. Take an arbitrary vertex $y=(y_1,\dots,y_d)\notin D$, so $y_1=1$. Then the vertex
$x=(0,y_2,\dots,y_d)\in D$
differs from $y$ in exactly one coordinate, hence $x$ and $y$ are adjacent in $Q_d$. Therefore $D$ dominates every vertex outside $D$. To verify the $1$-limited constraint, let $x\in D$. Among the $d$ neighbors of $x$, exactly one has first coordinate $1$, while the remaining $d-1$ neighbors have the first coordinate equal to $0$ and thus also lie in $D$.
Consequently, $x$ dominates exactly one vertex of $V(Q_d)\setminus D$, and the $1$-limited constraint holds.
Thus $\gammaenaL(Q_d)\le |D|=2^{d-1}$, and together with the lower bound we conclude
$\gammaenaL(Q_d)=2^{d-1}$.
\end{proof}

\section{Conclusion}
\label{sec:conclusion}

The notion of $k$-limited domination was introduced in \cite{podgorica1} as a restricted variant of classical domination, motivated by situations in which dominating vertices cannot serve arbitrarily many vertices outside the dominating set. The present paper shows that the parameter is not only well motivated, but also structurally rich and mathematically nontrivial.

This is reflected already in the computational complexity of the problem. While we proved that \textsc{$k$-Limited Dominating set} is $\mathsf{NP}$-complete for every fixed integer $k\ge 2$, the extremal case $k=1$ remains unresolved. Our reduction relies in an essential way on the assumption $k\ge 2$, and the corresponding forcing mechanism breaks down for $k=1$. This suggests that the case of $1$-limited domination requires a different approach, and leads naturally to the following question.

\begin{problem}
Is \textsc{$1$-Limited Dominating Set} $\mathsf{NP}$-complete?
\end{problem}

Another direction pursued in this paper concerns the behavior of $k$-limited domination under the Cartesian product. We established general lower and upper bounds for $\gammakL(G\square H)$, proved that both are sharp, and derived exact values for several natural graph families. In particular, for $k=1$ all exact values obtained in this paper attain the general upper bound from Theorem~\ref{CP-bounds}. This is not the typical behavior for arbitrary $k$: for example,
$\gamma^{\mathrm L}_2(K_2\square K_{1,3})=3$,
whereas Theorem~\ref{CP-bounds} yields only the upper bound $4$. Nevertheless, our results suggest that the extremal case $k=1$ may be special, which leads to the following question.

\begin{problem}
Is it true that for all connected graphs $G$ and $H$,
$$
\gamma^{\mathrm L}_1(G\square H)=
\min\big\{\gamma^{\mathrm L}_1(G)\,n(H),\,
\gamma^{\mathrm L}_1(H)\,n(G)\big\}\, ?
$$
\end{problem}

We conclude by discussing $k$-limited domination in hypercubes. For the classical domination number $\gamma(Q_d)$, a useful overview of the current state of knowledge is provided by the recent paper of Wu and Chen \cite{Wu}, whose appendix collects the best known lower and upper bounds for dimensions up to $33$. In addition, for dimensions divisible by $6$, their main theorem improves some of the previously best known lower bounds. The gray cells in Table~\ref{tabela3} reflect the corresponding part of this picture up to dimension $17$, while incorporating the mentioned improvement in the case $d=12$. The remaining entries are then arranged around these values in the setting of $k$-limited domination. Since for $k\ge d$ we have $\gamma_k^{\mathrm L}(Q_d)=\gamma(Q_d)$, an asterisk indicates that the corresponding value coincides with the one in the gray cell above it. The entries immediately above the gray cells represent the values of $\gamma_{d-1}^{\mathrm L}(Q_d)$, which are equal to the total domination number $\gamma_t(Q_d)$, and therefore to twice the domination number of the $(d-1)$-dimensional hypercube. The values in the first row follow from Proposition~\ref{1Lkock}.

For $d$-regular graphs, and in particular for hypercubes $Q_d$, the identity $\gamma_k^{\mathrm L}(Q_d)=\gamma_{1,d-k}(Q_d)$,
arising from Proposition~\ref{d-regular}, shows that the study of $k$-limited domination is naturally equivalent to the study of $(1,t)$-domination. As already noted, the case $t=1$ corresponds to total domination, while the cases $t\ge 2$ seem to remain largely open. In particular, the parameter $\gamma^{\mathrm L}_{d-2}(Q_d)$, equivalent to $(1,2)$-domination, appears to be a natural next direction for further research, especially in light of the recent interest in this parameter \cite{rija,fakhran,flower}.

\begin{table}[ht]
\centering
\caption{Known values and bounds for $\gammakL(Q_d)$.}
\label{tabela3}
\tiny
\setlength{\tabcolsep}{2.3pt}
\renewcommand{\arraystretch}{1}
\begin{tabular}{|c|*{17}{c|}}
\hline
 & $1$ & $2$ & $3$ & $4$ & $5$ & $6$ & $7$ & $8$ & $9$ & $10$ & $11$ & $12$ & $13$ & $14$ & $15$ & $16$ & $17$ \\
\hline
$\gamma^{\mathrm L}_1(Q_d)$
& \cellcolor{GoalGray}{$1$} & {$2$} & {$4$} & {$8$} & {$16$} & {$32$} & {$64$} & {$128$} & {$256$} & {$512$} & {$1024$} & {$2048$} & {$4096$} & {$8192$} & {$16384$} & {$32768$} & {$65536$} \\
\hline
$\gamma^{\mathrm L}_2(Q_d)$
& $*$ & \cellcolor{GoalGray}{$2$} & {$4$} &  &  &  &  &  &  &  &  &  &  &  &  &  &  \\
\hline
$\gamma^{\mathrm L}_3(Q_d)$
& $*$ & $*$ & \cellcolor{GoalGray}{$2$} & {$4$} &  &  &  &  &  &  &  &  &  &  &  &  &  \\
\hline
$\gamma^{\mathrm L}_4(Q_d)$
& $*$ & $*$ & $*$ & \cellcolor{GoalGray}{$4$} & {$8$} &  &  &  &  &  &  &  &  &  &  &  &  \\
\hline
$\gamma^{\mathrm L}_5(Q_d)$
& $*$ & $*$ & $*$ & $*$ & \cellcolor{GoalGray}{$7$} & {$14$} &  &  &  &  &  &  &  &  &  &  &  \\
\hline
$\gamma^{\mathrm L}_6(Q_d)$
& $*$ & $*$ & $*$ & $*$ & $*$ & \cellcolor{GoalGray}{$12$} & {$24$} &  &  &  &  &  &  &  &  &  &  \\
\hline
$\gamma^{\mathrm L}_7(Q_d)$
& $*$ & $*$ & $*$ & $*$ & $*$ & $*$ & \cellcolor{GoalGray}{$16$} & {$32$} &  &  &  &  &  &  &  &  &  \\
\hline
$\gamma^{\mathrm L}_8(Q_d)$
& $*$ & $*$ & $*$ & $*$ & $*$ & $*$ & $*$ & \cellcolor{GoalGray}{$32$} & {$64$} &  &  &  &  &  &  &  &  \\
\hline
$\gamma^{\mathrm L}_9(Q_d)$
& $*$ & $*$ & $*$ & $*$ & $*$ & $*$ & $*$ & $*$ & \cellcolor{GoalGray}{$62$} & {$124$} &  &  &  &  &  &  &  \\
\hline
$\gamma^{\mathrm L}_{10}(Q_d)$
& $*$ & $*$ & $*$ & $*$ & $*$ & $*$ & $*$ & $*$ & $*$ & \cellcolor{GoalGray}{$107$--$120$} & {$214$--$240$} &  &  &  &  &  &  \\
\hline
$\gamma^{\mathrm L}_{11}(Q_d)$
& $*$ & $*$ & $*$ & $*$ & $*$ & $*$ & $*$ & $*$ & $*$ & $*$ & \cellcolor{GoalGray}{180--192} & 360--384 &  &  &  &  &  \\
\hline
$\gamma^{\mathrm L}_{12}(Q_d)$
& $*$ & $*$ & $*$ & $*$ & $*$ & $*$ & $*$ & $*$ & $*$ & $*$ & $*$ & \cellcolor{GoalGray}{348--380} & 696--760  &  &  &  &  \\
\hline
$\gamma^{\mathrm L}_{13}(Q_d)$
& $*$ & $*$ & $*$ & $*$ & $*$ & $*$ & $*$ & $*$ & $*$ & $*$ & $*$ & $*$ & \cellcolor{GoalGray}{598--704} & 1196--1408 &  &  &  \\
\hline
$\gamma^{\mathrm L}_{14}(Q_d)$
& $*$ & $*$ & $*$ & $*$ & $*$ & $*$ & $*$ & $*$ & $*$ & $*$ & $*$ & $*$ & $*$ & \cellcolor{GoalGray}{1172--1408} & 2344--2816 &  &  \\
\hline
$\gamma^{\mathrm L}_{15}(Q_d)$
& $*$ & $*$ & $*$ & $*$ & $*$ & $*$ & $*$ & $*$ & $*$ & $*$ & $*$ & $*$ & $*$ & $*$ & \cellcolor{GoalGray}{$2048$} & {$4096$} &  \\
\hline
$\gamma^{\mathrm L}_{16}(Q_d)$
& $*$ & $*$ & $*$ & $*$ & $*$ & $*$ & $*$ & $*$ & $*$ & $*$ & $*$ & $*$ & $*$ & $*$ & $*$ & \cellcolor{GoalGray}{$4096$} & {$8192$} \\
\hline
$\gamma^{\mathrm L}_{17}(Q_d)$
& $*$ & $*$ & $*$ & $*$ & $*$ & $*$ & $*$ & $*$ & $*$ & $*$ & $*$ & $*$ & $*$ & $*$ & $*$ & $*$ & \cellcolor{GoalGray}{7419--8192} \\
\hline
\end{tabular}
\end{table}

From this perspective, filling the blank cells of Table~\ref{tabela3} is already an interesting problem in its own right. More specifically, it would be desirable to determine further exact values of $\gamma_k^{\mathrm L}(Q_d)$ and to gain a more systematic understanding of how these parameters behave as $k$ tends to $d$.

\vskip 1pc \noindent{\bf Acknowledgments.} 
This work has been supported by the European Commission's Horizon Europe Research and Innovation programme through the Marie Sk\l odowska-Curie Actions Staff Exchanges (MSCA-SE) under Grant Agreement no. 101182819 (COVER: (C)ombinatorial (O)ptimization for (V)ersatile Applications to (E)merging u(R)ban Problems).
The researcher was partially supported by Slovenian research agency ARIS, program no.~P1--0297 and project no.~J1--70016. 

\vskip 1pc \noindent{\bf Statements and Declarations.}
The author has no conflict of interest to declare.


\begin{thebibliography}{99}

\bibitem{azar}
J.~Azarija, M.~A.~Henning, S. Klav\v{z}ar,
(Total) Domination in Prisms Electron. J. Comb., 24(1), 2017, \#P1.19.
\url{https://doi.org/10.37236/6288}

\bibitem{podgorica1}
D.~Bo\v{z}ovi\'c, G.~Radi\'c, \v{Z}.~Kovijani\'c-Vuki\'cevi\'c, A.~Tepeh,
On $k$-limited domination in graphs,
Comp. Appl. Math. 46, 43 (2027). \url{https://doi.org/10.1007/s40314-026-03902-2}

\bibitem{BB}
B.~Bre\v{s}ar, P.~Dorbec, W.~Goddard, B.~L.~Hartnell, M.~A.~Henning, S.~Klav\v{z}ar, D.~F.~Rall,
Vizing's conjecture: a survey and recent results,
J. Graph Theory 69 (2012) 46--76.
\url{https://doi.org/10.1002/jgt.20565}

\bibitem{Chellali}
M.~Chellali, T.~W.~Haynes, S.~T.~Hedetniemi, A.~McRae,
$[1,2]$-sets in graphs,
Discrete Appl. Math. 161 (2013) 2885--2893.
\url{https://doi.org/10.1016/j.dam.2013.06.012}

\bibitem{rija}
R.~Erve\v{s}, A.~Tepeh, Induced cycles vertex number and $(1,2)$-domination in cubic graphs,  \emph{Appl. Math. Comput.}  {\bf 510} (2025), 129700,
\url{https://doi.org/10.1016/j.amc.2025.129700}.

\bibitem{fakhran}
M. H. Fakhran, A. A. Gorzin, M. A. Henning, A. Jafari and R. Touserkani, 
On $(1,2)$-domination in cubic graphs,
\emph{Discrete Math.} {\bf 344} (2021) 112546,
\url{https://doi.org/10.1016/j.disc.2021.112546}.

\bibitem{favaron}
O.~Favaron, M.~A.~Henning, J.~Puech, D.~Rautenbach,
On domination and annihilation in graphs with claw-free blocks,
Discrete Math. 231 (2001) 143--151.
\url{https://doi.org/10.1016/S0012-365X(00)00313-7}

\bibitem{FinkJacobson}
J.~F.~Fink, M.~S.~Jacobson,
$n$-domination in graphs,
in: Y.~Alavi, G.~Chartrand, O.~R.~Oellermann, A.~J.~Schwenk (Eds.),
Graph Theory with Applications to Algorithms and Computer Science,
Wiley, New York, 1985, pp.~283--300.


\bibitem{garey}
M.~R.~Garey, D.~S.~Johnson, Computers and Intractability; A Guide to the
Theory of NP-Completeness, W. H. Freeman \& Co., New York, 1979. 
\url{https://doi.org/10.1137/1024022}


\bibitem{GodHen18}
W.~Goddard, M.~A.~Henning,
A note on domination and total domination in prisms,
J. Comb. Optim. 35 (2018) 14--20.
\url{https://doi.org/10.1007/s10878-017-0150-0}

\bibitem{HarLiv}
F.~Harary, M.~Livingston,
Independent domination in hypercubes,
Appl. Math. Lett. 6 (1993) 27--28.
\url{https://doi.org/10.1016/0893-9659(93)90027-K}

\bibitem{HHH1}
T.~W.~Haynes, S.~T.~Hedetniemi, M.~A.~Henning (Eds.), Topics in Domination in Graphs, Developments in Mathematics, Vol. 64, Springer, Cham, 2020.
\url{https://doi.org/10.1007/978-3-030-51117-3}

\bibitem{HHH2}
T.~W.~Haynes, S.~T.~Hedetniemi, M.~A.~Henning (Eds.), Structures of Domination in Graphs, Developments in Mathematics, Vol. 66, Springer, Cham, 2021.
\url{https://doi.org/10.1007/978-3-030-58892-2}

\bibitem{HHH3}
T.~W.~Haynes, S.~T.~Hedetniemi, M.~A.~Henning, Domination in Graphs: Core Concepts, Series: Springer Monographs in Mathematics, Springer, Cham, 2023. 
\url{https://doi.org/10.1007/978-3-031-09496-5}

\bibitem{Peterin}
N.~Ghareghani, I.~Peterin, P.~Sharifani,
$[1,k]$-domination number of lexicographic products of graphs,
Bull. Malays. Math. Sci. Soc. 44 (2021) 375--392.
\url{https://doi.org/10.1007/s40840-020-00957-0}

\bibitem{flower}
G.~Radi\'c, G.~Radi\'c, A.~Tepeh, 
$(1,2)$-domination of Flower snarks, submitted.


\bibitem{at-dmgt}
A.~Tepeh,
Total domination in generalized prisms and a new domination invariant,
Discuss. Math. Graph Theory 41 (2021) 1165--1178.
\url{https://api.semanticscholar.org/CorpusID:209476345}

\bibitem{Wee}
G.~J.~M. van Wee, 
Improved sphere bounds on the covering radius of codes.
IEEE
Trans. Inform. Theory 34 (1988) 237--245.
\url{https://doi.org/10.1109/18.2632}

\bibitem{Wu}
Y.-S.~Wu, J.-Y.~Chen,
Improved lower bounds on the domination number of hypercubes and binary codes with covering radius one,
Discrete Math. 347 (2024) 113752.
\url{https://doi.org/10.1016/j.disc.2023.113752}

\end{thebibliography}
\end{document}